%% file: main.tex
\documentclass[12pt]{article}
\usepackage{graphicx}
\newcommand{\bt}{\begin{tikzcd}}
\newcommand{\et}{\end{tikzcd}}
\usepackage[numbered,framed]{matlab-prettifier}
\usepackage{enumitem}
\usepackage{comment}
\usepackage{amsthm}
\usepackage{amsmath, amscd}
\usepackage{tgtermes}
\theoremstyle{definition}
\usepackage{array}
\makeatletter
\newcommand{\thickhline}{\noalign {\ifnum 0=`}\fi \hrule height 1pt \futurelet \reserved@a \@xhline}
\newcolumntype{"}{@{\hskip\tabcolsep\vrule width 1pt\hskip\tabcolsep}}
\makeatother
\usepackage{hyperref}
\usepackage{tikz-cd}
\usepackage{xurl}
\usepackage{float}
\usepackage{amssymb}
\usepackage{mathtools}
\usepackage{mathrsfs}
\usepackage{calligra}
\usepackage[left=3cm,right=3cm, top = 3cm, bottom = 3cm]{geometry}
\usepackage{bbm}

\newcommand{\M}{\mathcal{M}}
\newcommand{\A}{\mathcal{A}}
\newcommand{\C}{\mathbb{C}} 
\newcommand{\Z}{\mathbb{Z}} 
\newcommand{\pr}{\mathbb{P}}

\newcommand{\Mct}{\M^{ct}}

\newcommand{\ul}{\underline} 
\newcommand{\ol}{\overline}

\newcommand{\HH}{\mathbb{H}}
\tikzset{labl/.style={anchor=south, rotate=-90, inner sep=.5mm}}
\newtheorem{theorem}{Theorem}[section]
\newtheorem{lemma}[theorem]{Lemma}
\newtheorem{proposition}[theorem]{Proposition}

\newtheorem{definition}[theorem]{Definition}
\newtheorem{construction}[theorem]{Construction}

\newtheorem{remark}[theorem]{Remark}
\newtheorem{example}[theorem]{Example}

\usepackage{epigraph}
\usepackage{hyperref}
\usepackage{appendix}
\usepackage{caption}
\usepackage{mathtools}

\newcommand{\thistheoremname}{}
\newtheorem*{genericthm*}{\thistheoremname}
\newenvironment{namedtheorem*}[1]
  {\renewcommand{\thistheoremname}{#1}
   \begin{genericthm*}}
  {\end{genericthm*}}

\title{The fiber product of the Torelli map with any product $\A_{g_1}\times \dots \times \A_{g_k}\to\A_g$ is reduced}
\author{Lycka Drakengren}
\date{}

\begin{document}

\maketitle

\begin{abstract}
    We prove that the fiber product of the Torelli map $t\colon \Mct_g \to \A_g$ with any product $\A_{g_1}\times\dots\times \A_{g_k} \to \A_g$ for $g=g_1+\dots+g_k$ has a reduced scheme structure. As a consequence, letting $d=\text{codim}(t^*[\A_{g_1}\times\dots\times \A_{g_k}])$, we find that the class $t^*[\A_{g_1}\times\dots\times \A_{g_k}]\in \mathsf{CH}^{d}(\Mct_g)$ is tautological. In particular, we obtain ${t^*[\A_{g_1}\times\dots\times \A_{g_k}] = 0}$ for $d > 2g-3.$
\end{abstract}

\section{Introduction}

\subsection{Moduli of decomposable abelian varieties and the Torelli map}

The moduli space $\A_g$ of \textit{principally polarized varieties} of dimension $g$ is a nonsingular Deligne-Mumford stack obtained over $\mathbb{C}$ as an orbifold quotient $$\A_g = [\text{Sp}(2g,\Z)\backslash\HH_g],$$
\noindent where $\HH_g$, the \textit{Siegel upper half-space}, is the space of matrices $$\HH_g = \{\Pi \in M_{g\times g}(\C)| \Pi^T = \Pi, \text{Im}(\Pi) >0\}.$$

For any $g=g_1+\dots+g_k$, we have a product map \begin{equation*}
\begin{aligned}
&\pi\colon  \A_{g_1}\times\dots\times \A_{g_k} \to \A_g\\
&(X_1,\dots,X_k) \mapsto X_1\times \dots \times X_k.
\end{aligned}
\end{equation*}

\noindent Pushing forward the fundamental class of $\A_{g_1}\times \dots \times \A_{g_k}$ under $\pi$, we obtain a cycle $$[\A_{g_1}\times \dots \times \A_{g_k}]\in \mathsf{CH}^*(\A_g).$$

The intersection theory of $\A_g$ can be studied using the moduli space $\Mct_g$ of \textit{compact type curves}, parametrizing stable curves $C$ whose dual graph $\Gamma_C$ is a tree. Compact type curves form an open substack $\Mct_g\subset \overline{\M}_g$ of the moduli space of stable curves of genus $g$, here considered over $\C$. The principally polarized \textit{Jacobian} $J(C)$ of a compact type curve $C$ parametrizes line bundles of multidegree $0$, where the polarization is given by the theta divisor. For $[C]\in \overline{\M}_g$, the Jacobian $J(C)$ is proper precisely when $[C]\in \Mct_g$. The \textit{Torelli map} is the morphism
\begin{equation*}
\begin{aligned}
t\colon &\Mct_g \to \A_g\\
&[C] \mapsto [J(C)],
\end{aligned}
\end{equation*}
whose pullback $t^*\colon \mathsf{CH}^*(\A_g) \to \mathsf{CH}^{*}(\Mct_g)$ gives us a tool for investigating relations in $\mathsf{CH}^*(\A_g)$.

To study the pullback $t^*[\A_{g_1}\times\dots \times \A_{g_k}]$ using excess intersection theory \cite{Ful}, we need to understand the precise scheme structure for the fiber product \begin{equation*}
\begin{aligned}
\bt Y \ar[r]\ar[d]& \A_{g_1}\times\dots\times \A_{g_k} \ar[d,"\pi"]\\
\Mct_g\ar[r,"t"]&\A_g. \et 
\end{aligned}
\end{equation*}

\noindent To study the local structure of $Y$, we will find it useful to consider a lift of the product map to $\HH_g$ given by
\begin{equation*}
\begin{aligned}
&\pi\colon  \HH_{g_1}\times\dots\times \HH_{g_k} \to \HH_g\\
&(\Pi_1,\dots, \Pi_k)\mapsto 
\begin{pmatrix}
    \Pi_1 &  0 & 0\\
0  & \ddots & 0 \\
0   & 0 & \Pi_k
\end{pmatrix}_.
\end{aligned}
\end{equation*}

\noindent Locally, we can also lift the Torelli map to $\HH_g$ as follows: For $[C]\in \Mct_g$, choose a symplectic basis $\{A_i,B_i\}_{1\leq i \leq g}$ for $H_1(C,\Z)$ and a basis $\{\omega_i\}_{1\leq i \leq g}$ for $H^0(\omega_C)$ normalized by ${\int_{A_j}\omega_i = \delta_{i,j}}$. The period matrix $$\Pi_{i,j} = \int_{B_j}\omega_i$$ gives a lift of $[J(C)]$ to $\HH_g$ \cite[\S 11.1]{BL}.

\subsection{Tautological classes and loci of decomposable abelian varieties}
Let $\bt \pi \colon \mathcal{X}_g \to \A_g \et$ be the universal abelian variety and $\mathbb{E} = \pi_*\Omega_\pi$ the \textit{Hodge bundle} on $\A_g$. The Chern classes $\lambda_i = c_i(\mathbb{E})$ generate a subring $R^*(\A_g) \subset \mathsf{CH}^*(\A_g)$ referred to as the \textit{tautological ring} \cite{vdg}. For $\Mct_g$, we instead define the tautological ring $R^*(\Mct_g)$ to be the $\mathbb{Q}$-subalgebra of $\mathsf{CH}^{*}(\Mct_g)$ generated by pushforwards from boundary strata of products of kappa- and psi classes \cite{GP}. In both cases, elements of the tautological ring are referred to as \textit{tautological classes}. Given a cycle in $\mathsf{CH}^*(\A_g)$, it is natural to ask whether or not it is tautological. For $\alpha \in \mathsf{CH}^k(\A_g)$, the \textit{tautological projection} $\text{taut}(\alpha)$ is the unique class in $R^k(\A_g)$ satisfying $$\int_{\overline{\A}_g}\alpha \cdot\beta\cdot\lambda_g = \int_{\overline{\A}_g}\text{taut}(\alpha) \cdot\beta\cdot\lambda_g \hspace{2pt}\text{ for all }\hspace{2pt} \beta\in \mathsf{R}^{\binom{g}{2}-k}(\A_g),$$ where $\ol{\A}_g$ is any choice of toroidal compactification of $\A_g$ \cite{CMOP}. To prove that $\alpha \not \in R^*(\A_g)$, it is enough to show that $t^*\alpha \neq t^*\text{taut}(\alpha)$ in $\mathsf{CH}^{*}(\Mct_g)$.

The classes $[\A_1\times \A_{g-1}]\in \mathsf{CH}^{g-1}(\A_g)$ have been studied in \cite{COP} and \cite{NL}. In \cite{NL}, Iribar López shows that $[\A_1\times \A_{g-1}]$ is nontautological for $g=12$ or $g\geq 16$ with $g$ even. In \cite{COP}, Canning, Oprea and Pandharipande compute the Torelli pullback $t^*[\A_1\times \A_{g-1}]$ using excess intersection theory on the fiber product of $t\colon \Mct_g\to \A_g$ and $\pi\colon \A_1\times \A_{g-1} \to \A_g$. This method uses the fact that the fiber product is reduced, which is proven in \cite{COP} using deformation theory of stable maps. By computing the class $t^*[\A_1\times \A_5]$ and comparing it with ${t^*\text{taut}([\A_1\times \A_5])}$, they are able to show that $[\A_1\times \A_5]\in \mathsf{CH}^5(\A_6)$ is nontautological. 

Our main result of this paper, Theorem \ref{mainthm}, opens up the possibility of calculating the classes $t^*[\A_{g_1}\times\dots\times \A_{g_k}]$ for any $g=g_1+\dots + g_k$. In \cite{inprep}, we give a recursive formula for computing these classes by adapting methods from \cite{COP}. The fiber product $Y$ consists of several intersecting components whose dimensions can exceed $\text{dim}(t^*[\A_{g_1}\times\dots\times \A_{g_k}])$. The first method presented in \cite{inprep} for calculating $t^*[\A_{g_1}\times\dots\times \A_{g_k}]$ involves stratifying $Y$, determining the normal bundles associated to containments of closed strata and using suitable local models to decompose the class $t^*[\A_{g_1}\times\dots\times \A_{g_k}]$ into canonical contributions from the closed strata.

\begin{theorem}\label{mainthm}
    The fiber product of the Torelli map $t\colon \Mct_g \to \A_g$ with the product map $\pi\colon  \A_{g_1}\times\dots\times \A_{g_k} \to \A_g$ for $g=g_1+\dots+g_k$ is reduced.
\end{theorem}

Theorem \ref{mainthm} is proven in Section \ref{1.1}. We use local series expansions for the Torelli map, derived from \cite{HN}, to find the precise scheme structure of the fiber product of $t$ and $\pi$. 

\subsection{The fiber product of the Torelli map with a general product locus}\label{fibprod}

In order to state the precise local equations for the fiber product $Y$, we first describe the $\C$-points of $Y$. By definition, a $\C$-point is given by a pair $([C],J)$ where $[C]\in \Mct_g$ and ${J=(X_1,\dots,X_k)\in \A_{g_1}\times\dots\times \A_{g_k}}$, together with a choice of isomorphism $$\bt\iota \colon J(C) \hspace{-5pt} \ar[r,shorten >=5pt, shorten <=5pt,"\sim"]&\hspace{-5pt} X_1\times \dots \times X_k.\et$$

\noindent We describe how the $\C$-points $([C],J,\hspace{-5pt}\bt\iota \colon J(C) \hspace{-5pt} \ar[r,shorten >=5pt, shorten <=5pt,"\sim"]&\hspace{-5pt} \pi(J)\hspace{-5pt}\et)$ correspond to tuples $([C],J,\sigma)$, where $$\sigma\colon V(\Gamma_C)_{>0} \to \{1,\dots,k\},$$\noindent for $V(\Gamma_C)_{>0}=\{v\in V(\Gamma_C)|g(v)>0\}$, is a $k$-coloring such that $\sum_{v \in \sigma^{-1}(j)} g(v) = g_j$ for $j=1,\dots,k$. Let $\{C_v\}_{v\in V(\Gamma_C)}$ be the irreducible components of $[C]\in \Mct_g$. Since $C$ is of compact type, we can decompose $J(C) = \prod_{v\in V(\Gamma_C)_{>0}} J(C_v)$. A principally polarized abelian variety has a unique decomposition into a product of indecomposable principally polarized abelian varieties \cite[Corollary 3.23]{CG}. The isomorphism $\iota$ is thus determined by a partition of $V(\Gamma_C)_{>0}$ into sets $S_1,\dots,S_k$ such that $\iota$ induces isomorphisms $\bt \prod_{v\in S_j}J(C_v)\hspace{-5pt} \ar[r,shorten >=5pt, shorten <=5pt,"\sim"]&\hspace{-5pt}X_j \et$ for $j=1,\dots,k$. We associate a $k$-coloring $\sigma$ to $\iota$ via $\sigma(v) = j$ for $v\in S_j$. Moreover, we have $$J = \big([\hspace{1pt}\prod_{v\in \sigma^{-1}(1)}J(C_v)],\dots,[\hspace{1pt}\prod_{v\in \sigma^{-1}(k)}J(C_v)]\big) \in \A_{g_1}\times\dots\times \A_{g_k}.$$ We henceforth denote the $\mathbb{C}$-points of $Y$ by tuples $([C],J,\sigma)$ as above.

\begin{definition}
For $([C], J, \sigma)\in Y$, we say that a path $\gamma$ (with no repeated vertices) of the dual graph $\Gamma_C$ is \textit{critical} if it satisfies the following properties:
\begin{itemize}
    \item The end vertices $v,v'$ of $\gamma$ are of positive genus
    \item The interior vertices of $\gamma$ are of genus $0$
    \item $\sigma(v) \neq \sigma(v')$
\end{itemize}
\noindent We denote the set of critical paths associated to $([C], J, \sigma)\in Y$ by $P_\sigma$.
\end{definition}

Figure \ref{critfig} shows an example of a point $([C], J,\sigma)\in Y$ for $k=2$ with the associated critical paths of $\Gamma_C$.

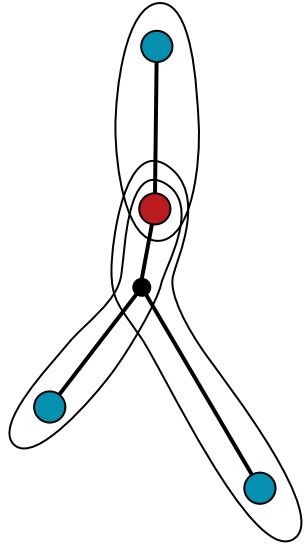
\begin{figure}[h]
\centering
\captionsetup{width=0.8\linewidth}
\scalebox{1}{\input{criticalpaths.tex}}
\caption{
Critical paths in $\Gamma_C$ associated to $([C],J,\sigma)\in Y$ in the case $k=2$. The red resp. blue vertices correspond to irreducible components $C_v$ of $C$ with $\sigma(v) = 1$ resp. $2$. The black vertices correspond to curves of genus $0$. The encircled paths are the critical paths of $\Gamma_C$.
}\label{critfig}
\end{figure}

In order to describe the scheme structure at a point $([C],J,\sigma)\in Y$, we would like to describe the maps $t$ and $\pi$ in suitable local coordinates. For this purpose, we view $[J(C)]$ and $J$ as elements of $\mathbb{H}_g$ and $\mathbb{H}_{g_1}\times\dots\times \mathbb{H}_{g_k}$ respectively, using a local lift of the Torelli map. The product map $\pi$ is straightforward to describe in local coordinates, see Section \ref{loccoord}. It is more difficult to describe the Torelli map $t$ in terms of the local coordinates $\ul{x},\ul{s}$ near $[C]\in \Mct_g$, where $\ul{x}$ are coordinates corresponding to the boundary stratum associated to $\Gamma_C$, and $\ul{s} = (s_e)_{e\in E(\Gamma_C)}$ are smoothing parameters as in Construction \ref{plumbing}. To do this, we need to be able to extend holomorphic differentials of $C$ over a neighborhood around $[C]$ in $\Mct_g$, and describe how they vary in terms of $\ul{x},\ul{s}$. We do this by adapting results from \cite{HN}. Precisely, given a holomorphic differential $\Omega$ on $C$, Hu and Norton associate a curve $C_{\ul{x},\ul{s}}$ to $\ul{x},\ul{s}$ (for sufficiently small $|\ul{x}|,|\ul{s}|$) and construct an explicit, suitably normalized, holomorphic differential $\Omega_{\ul{x},\ul{s}}$ on $C_{\ul{x},\ul{s}}$ which varies holomorphically in $\ul{x},\ul{s}$ and restricts to $\Omega$ on ${C_{\ul{0},\ul{0}} = C}$. Moreover, they describe a procedure for finding the power series expansion of $\Omega_{\ul{x},\ul{s}}$ with respect to the smoothing parameters $\ul{s}$. Analyzing this procedure and integrating the differentials $\Omega_{\ul{x},\ul{s}}$ over cycles of $H_1(C_{\ul{x},\ul{s}},\Z)$, we find properties of $t$ in the coordinates $\ul{x},\ul{s}$ that allow us to compute the explicit scheme structure at $([C],J,\sigma)$.

We now state a refined version of Theorem \ref{mainthm}, describing the precise local scheme structure for $Y$.

\begin{theorem}\label{refinedthm}
     Let $\ul{x}$ together with $\ul{s}=(s_e)_{e\in E(\Gamma_C)}$ be the coordinates of a small analytic neighborhood of $[C]$ in $\Mct_g$, where $\ul{x}$ correspond to coordinates around $[C]$ in the boundary stratum of $\Mct_g$ associated to $\Gamma_C$, and $s_e$ corresponds to a smoothing of the node associated to the edge $e$ as in Construction \ref{plumbing}. The coordinate ring near $([C],J,\sigma)\in Y$ is given by $$\frac{\C[[\ul{x},\ul{s}]]}{(\prod_{e\in \gamma}s_e)_{\gamma \in P_\sigma}}.$$\noindent In particular, it is reduced.
\end{theorem}

\subsection{Torelli pullbacks of product loci}\label{pullbacks}

Many classes $t^*[\A_{g_1}\times\dots\times \A_{g_k}]\in \mathsf{CH}^*(\Mct_g)$ vanish for dimension reasons. Indeed, we have $\text{dim}(\Mct_g)=3g-3$ and $\text{dim}(\A_g)=\frac{g(g+1)}{2}$. The codimension $d$ of $\pi$ equals $\sum_{1\leq i <j \leq k} g_ig_j$. In other words, we have $$\text{codim}(t^*[\A_{g_1}\times\dots\times \A_{g_k}])= \sum_{1\leq i <j \leq k} g_ig_j.$$ We deduce that $t^*[\A_{g_1}\times\dots\times \A_{g_k}]=0$ for $d>3g-3$. If we additionally know that the class  $t^*[\A_{g_1}\times\dots\times \A_{g_k}]$ is tautological, we find that the class vanishes for $d>2g-3$ since $R^{>2g-3}(\Mct_g) = 0$ \cite{GV}. As an application of Theorem \ref{mainthm}, we prove the following result in Section \ref{a3}:
\begin{theorem}\label{vanishing}
   The class $t^*[\A_{g_1}\times\dots\times \A_{g_k}]\in \mathsf{CH}^*(\Mct_g)$ is tautological for all decompositions $g=g_1+\dots+g_k$. Thus, $t^*[\A_{g_1}\times\dots\times \A_{g_k}] = 0 \in \mathsf{CH}^{d}(\Mct_g)$ for $d > 2g-3.$
\end{theorem}

We have the following possibilities for $g_1,\dots,g_k$ such that $d\leq 2g-3$, meaning that ${t^*[\A_{g_1}\times\dots\times \A_{g_k}]}$ is possibly nonzero (see Section \ref{a3}):

\begin{proposition}\label{tuples}
The tuples $(g_1,\dots,g_k)$ with $1\leq g_1\leq \dots \leq g_k$ and $\sum_{1\leq i \leq k}g_i = g$ such that $d= \text{codim}(t^*[\A_{g_1}\times\dots\times \A_{g_k}])\leq 2g-3$ are $(1,g-1),(1,1,g-2),(2,g-2)$ and $(3,3)$. 
\end{proposition}

In particular, Proposition \ref{tuples} implies that $t^*[\A_{3}\times \A_{g-3}] = 0$ in $\mathsf{CH}^{3g-9}(\Mct_g)$ for $g\geq 7$. 

\subsection{Further directions}
It is possible that the proof of Theorem \ref{mainthm} can be adapted to different settings, which would be interesting to explore. Let $\mathcal{P}\colon \widetilde{\mathcal{R}}_{g+1}\to \A_g$ be the extended Prym map satisfying the assumptions of \cite[Lemma 5.1]{Beauville}. One could adapt our methods for proving Theorem \ref{mainthm}, to investigate the following:
\begin{itemize}
   \item[]\textbf{Question 1:} Is the fiber product of the Prym map $\mathcal{P}\colon \widetilde{\mathcal{R}}_{g+1}\rightarrow \A_{g}$ and the product map $\pi\colon \A_{g_1}\times\dots\times\A_{g_k}\to \A_g$ reduced?
\end{itemize}

For $d>0$, the \textit{Noether-Lefschetz locus} $\widetilde{\text{NL}}_{g,d}\subset \A_g$ is defined as

$$\widetilde{\text{NL}}_{g,d}= \Biggl\{f\colon E \to (X,\theta) \Bigg| \begin{aligned}
&\hspace{5pt} (X,\theta)\in \A_g, E \text{ an elliptic curve,} \\ &\hspace{5pt}\text{$f$ a homomorphism with }\text{deg}(f^*\theta) = d\end{aligned}\Biggr\}.$$

\vspace{5pt}

In \cite{GL}, Greer and Lian prove that the fiber product of the Torelli map $\Mct_{g,1}\to \A_g$ and the map $\widetilde{\text{NL}}_{g,d}\to \A_g$ corresponds to a moduli space of stable maps $\mathcal{M}_{g,1}^{ct,q}(\mathcal{E},d)$ to a varying elliptic curve (see \cite{GL} for notations). This relates the intersection theory of $\A_g$ to the Gromov-Witten theory of these stable maps, which is in turn related to the genus $1$ Gromov-Witten theory of $\text{Hilb}^n(\C^2)$ \cite{ILPT}. To better understand the intersection theory, we would find it useful to answer the following question:
\begin{itemize}
   \item[]\textbf{Question 2:} Given $d>0$, what is the scheme structure of $\mathcal{M}_{g,1}^{ct,q}(\mathcal{E},d)$?
\end{itemize}

\subsection*{Acknowledgements}

I thank Rahul Pandharipande for his interest in this result and its applications. I also thank Samir Canning, Jeremy Feusi and Aitor Iribar López for useful comments. The author is supported by SNF-200020-219369.

\section{Reducedness of the fiber product}\label{redsection}

In Section \ref{loccoord}, we construct suitable analytic neighborhoods and coordinates near ${[C]\in \Mct_g}$, $[J(C)]\in \mathbb{H}_g$ and $J\in \mathbb{H}_{g_1}\times\dots\times \mathbb{H}_{g_k}$. In Section \ref{term}, we introduce relevant terminology for proving Theorem \ref{refinedthm}, mostly adapted from \cite{HN}, \cite{GKN}. In Section \ref{1.1}, we analyze the expansion of $t$ in the given coordinates. We conclude Section \ref{1.1} with a proof of Theorem \ref{refinedthm}.

\subsection{Local coordinates}\label{loccoord}
Let $W$ be an analytic neighborhood around $[J(C)] \in \mathbb{H}_g$ such that $h W \cap W = \emptyset$ for ${h\in \text{Sp}(2g,\mathbb{Z})}$ with $h [J(C)] \neq [J(C)]$. We choose coordinates $\ul{y} = ({y_{i,j}})_{1\leq i \leq j \leq g}$ on $W$ such that the symmetric matrix $\Pi_{\ul{y}}\in W$ associated to $\ul{y}$ has entries ${(\Pi_{\ul{y}})}_{i,j} = y_{i,j}$ for $1\leq i \leq j \leq g$. Denote by $V$ the preimage of $W$ under the lifted map $\pi \colon \mathbb{H}_{g_1}\times\dots\times \mathbb{H}_{g_k} \to \mathbb{H}_{g}$. Let the coordinates on the factor $\mathbb{H}_{g_i}$ be $\ul{x}^{(i)} = (x^{(i)}_{m,n})_{1\leq  m\leq n \leq g_i}$ such that the entries of the associated symmetric matrix $\Pi_{\ul{x}^{(i)}}\in \mathbb{H}_{g_i}$ are given by $(\Pi_{\ul{x}^{(i)}})_{m,n} = x^{(i)}_{m,n}$ for $1\leq m \leq n \leq g_i$. We also use the notation $(\ul{x}^{(i)})_{1\leq i \leq k}$ for the corresponding coordinates on $V$. We can assume that the map $\pi$ is described in these coordinates by
$$(\Pi_{\ul{x}^{(1)}},\dots, \Pi_{\ul{x}^{(k)}})\mapsto 
\begin{pmatrix}
    \Pi_{\ul{x}^{(1)}} &  0 & 0\\
0  & \ddots & 0 \\
0   & 0 & \Pi_{\ul{x}^{(k)}}
\end{pmatrix}_{\textstyle .}
$$

Let $U\subset \Mct_g$ be a sufficiently small analytic neighborhood around $[C]$ with $t(U)\subset W$ so that we can find coordinates $(\ul{x},\ul{s})$ for $U$ as in \cite[Remark 2.4]{HN}:

\begin{itemize}
    \item For each $v\in V(\Gamma_C)$, the coordinates $x_{v,1},\dots, x_{v,3g(v)-3 + n(v)}$ correspond to coordinates in a neighborhood of the pointed curve $[(C_v,q_1,\dots,q_{n(v)})]\in \M_{g(v),n(v)}$, where $q_1,\dots,q_{n(v)}$ are the nodes in $C$ lying on $C_v$. The coordinates $$\ul{x} = (x_{v,1},\dots, x_{v,3g(v)-3 + n(v)})_{v\in V(\Gamma_C)}$$ can be viewed as coordinates for a neighborhood of $[C]$ in the boundary stratum \newline $\Mct_{\Gamma_C} = \prod_{v\in V(\Gamma_C)} \Mct_{g(v),n(v)}$ of $\Mct_g$ associated to $\Gamma_C$.

    \item For each edge $e$ in the dual graph $\Gamma_C$ of $C$, the coordinate $s_e$ corresponds to a smoothing of the associated node as in Construction \ref{plumbing}. We write $\ul{s} = (s_e)_{e\in E(\Gamma_C)}$.
\end{itemize}

\subsection{Terminology}\label{term}

The notation in this section is mostly adapted from \cite{HN} and \cite{GKN}.

For $[C]\in \Mct_g$, consider an oriented edge $e$ in the dual graph $\Gamma_C$. The oppositely oriented edge is denoted $-e$. The source vertex of $e$ is denoted $v(e)$ and corresponds to an irreducible component $C_{v(e)}$ of $C$. We denote by $q_e$ the preimage in $C_{v(e)}$ of the node $q_{|e|}\in C$ associated to $e$. For $v\in V(\Gamma_C)$, the set of directed edges $e$ such that $v = v(e)$ is denoted $E_v$. Let $z_e$ be a local coordinate in a neighborhood $V_e\subset C$ of $q_e$ so that $z_e\colon V_e \to D(0,1)\subset \mathbb{C}$ is an isomorphism and $z_e(q_e) = 0$. We also assume that the $V_e$ are pairwise disjoint.

\begin{construction}\label{plumbing}
    For $[C]\in \Mct_g$ and $\ul{s} = (s_e)_{e\in E(\Gamma_C)}\in \C^{|E(\Gamma_C)|}$ with $|s_e|<1$, the \textit{standard plumbing} $C_{\ul{s}}$ is obtained as follows \cite[Definition 4.1]{GKN}: For each $e\in E(\Gamma_C)$, let $U_e$ be the open $\sqrt{|s_e|}$-neighborhood of $q_e$. Denote the positively oriented boundary of $U_e$ by $\gamma_e$. The map $I_e\colon \gamma_e \to \gamma_{-e}$ is defined by $z_e\mapsto s_e/z_{-e}$. We define $$C_{\ul{s}} = (C\backslash \sqcup_{e\in E(\Gamma_C)} U_e)/(\gamma_e\sim_{I_e} \gamma_{-e})_{e\in E(\Gamma_C)}.$$ The holomorphic structure of $C_{\ul{s}}$ is obtained from that of $C\backslash \sqcup_{e\in E(\Gamma_C)} U_e$ via the identification of variables $z_e =  s_e/z_{-e}$ in the charts $(V_e,z_e), (V_{-e},z_{-e})$.
\end{construction}

\begin{definition}
Given a smooth curve $C'$, a fixed point $q_0\in C'$ and a symplectic basis $\{A_i,B_i\}_{1\leq i \leq g(C')}$ for $H_1(C',\mathbb{Z})$, the \textit{A-normalized Cauchy kernel} $K_{C'}$ is uniquely defined by the following properties:
\begin{itemize}
    \item $K_{C'}(p,q)$ is a meromorphic differential in $p\in C'$, with only simple poles at $p=q$ and $p=q_0$, with residues $\pm \frac{1}{2\pi i}$
    \item $\int_{p\in A_i}K_{C'}(p,q) = 0$ for $i=1,\dots, g(C')$
\end{itemize}
\end{definition}

\begin{example}
    The Cauchy kernel of $\pr^1$ with poles at $p=q$ and $p=\infty$ is given by $$K_{\pr^1}(p,q) = \frac{1}{2\pi i} \frac{dp}{p-q}.$$
\end{example}

\begin{definition}
The \textit{$A$-normalized fundamental bidifferential} $b_{C'}$ is defined by $$b_{C'}(p,q) = 2\pi i\cdot  d_qK_{C'}(p,q),$$ \noindent where $d_q$ denotes the exterior derivative with respect to $q$.
\end{definition}

We denote by $K_v$ resp. $b_v$ the Cauchy kernel resp. fundamental bidifferential associated to an irreducible component $C_v$ of a curve $[C]\in \Mct_g$, where $v\in V(\Gamma_C)$.

The regular part $\mathbf{K}_v$ of $K_v$ on $\sqcup_{e\in E_v} V_e\times \sqcup_{e\in E_v} V_e$ satisfies 

\begin{equation*}
\mathbf{K}_v(z_e,\zeta_{e'}) =
\begin{cases}
    \begin{aligned}
        &K_v(z_e,\zeta_{e'}) \hspace{8pt}&\text{if }& e\neq e'\\
        &K_v(z_e,\zeta_{e'})-\frac{dz_e}{2\pi i (z_e-\zeta_{e'})} \hspace{8pt}&\text{if }& e= e',
    \end{aligned}
\end{cases}
\end{equation*}
\noindent where we use the notations $p\in V_e$ and $z_e = z_e(p)$ for the entries of $K_v, \mathbf{K}_v$ interchangeably, and where $\zeta_e$ is the local coordinate in $V_e$, denoted differently from $z_e$ to distinguish the local coordinates in the different entries of $K_v$.

For $\boldsymbol{b}_{v}(z_e,\zeta_{e'}) = 2\pi i\cdot  d_{\zeta_{e'}}\mathbf{K}_{v}(z_e,\zeta_{e'})$, we have

\begin{equation*}
\boldsymbol{b}_v(z_e,\zeta_{e'}) =
\begin{cases}
    \begin{aligned}
        &b_v(z_e,\zeta_{e'}) \hspace{8pt}&\text{if }& e\neq e'\\
        &b_v(z_e,\zeta_{e'})-\frac{dz_ed\zeta_{e'}}{(z_e-\zeta_{e'})^2} \hspace{8pt}&\text{if }& e= e'.
    \end{aligned}
\end{cases}
\end{equation*}

\noindent We write $b_v(z_e,q_{e'}) = \tilde{b}_v(z_e,q_{e'})dz_e$, where $b_v(z_e,\zeta_{e'}) = \tilde{b}_v(z_e,\zeta_{e'})dz_ed\zeta_{e'}$, for the differential obtained by evaluating the fundamental bidifferential at a node in the second argument, and similarly for $\boldsymbol{b}_v$. For $e,e'\in E_v$, we write $\beta_{e,e'} = \tilde{\boldsymbol{b}}_v(q_e,q_{e'})$.

Let $[C_{\ul{x},\ul{s}}] \in U$ be the curve associated to the coordinates $\ul{x} = (x_{v,1},\dots, x_{v,3g(v)-3+n})_{v\in V(\Gamma_C)}$ and $\ul{s} = ({s_e})_{e \in E(\Gamma_C)}$. Assume that we want to vary a differential $\Omega \in H^0(\omega_{C})$ to obtain a holomorphic differential $\Omega_{\ul{x},\ul{s}} \in H^0(\omega_{C_{\ul{x},\ul{s}}})$ with $\Omega_{\ul{0},\ul{0}} = \Omega$. The methods in \cite{HN} depend analytically on the variables $\ul{x}$, see \cite[Remark 2.4]{HN}. Thus, we can without loss of generality fix $\ul{x}=\ul{0}$ and consider the differentials $\Omega_{\ul{s}} = \Omega_{\ul{0},\ul{s}}$ on the curves $C_{\ul{s}} = C_{\ul{0},\ul{s}}$. The restriction of $\Omega$ to $C_v$ is denoted $\Omega_v$.

For $e\in E(\Gamma_C)$, $r\geq 0$, we inductively define holomorphic differentials $\xi^{(r)}_e$ on $V_e$ by
\begin{equation}\label{xis}
\begin{cases}
\begin{aligned}
    \xi^{(0)}_{e}(z_e) &= \Omega_{v(e)}(z_e) \\
    \xi^{(r)}_{e}(z_e) &= \sum_{e'\in E_v}\int_{\zeta_{e'}\in \gamma_{e'}}\mathbf{K}_{v(e)}(z_e,\zeta_{e'})\cdot I^*_{e'}\xi_{-e'}^{(r-1)}(\zeta_{e'})\ \text{ if } r>0.\\
\end{aligned}
\end{cases}
\end{equation}
That $\xi^{(0)}_{e}(z_e)$ is holomorphic follows since $C$ is of compact type. Write $\xi^{(r)}_{e}(z_e) =  \tilde{\xi}^{(r)}_{e}(z_e)dz_e$ and $\tilde{\xi}_e= \tilde{\xi}^{(0)}_{e}(q_e)$.

The collection of all oriented paths (possibly with repeated edges) of length $r$ in $\Gamma_C$ starting at a vertex $v$ is denoted $L_v^r$. For a path $l = (e_1,\dots,e_r) \in L_v^r$, write $\beta(l)=\prod_{j=1}^{r-1}\beta_{-e_j,e_{j+1}}$ and $s(l)=\prod_{i=1}^{r}s_{e_i}$. Let $|\cdot|_\infty$ denote the norm $|(s_1,\dots,s_r)|_\infty =\text{sup}_{i}|s_i|$.

Given a symplectic basis $\{A_i,B_i\}_{1\leq i \leq g(v)}$ for $H_1(C_{v},\mathbb{Z})$, differentials $\eta_v = \sum_{r\geq 1}\eta_v^{(r)}$ are defined on $C_v\backslash (\sqcup_{e\in E_v} U_e)$ by
\begin{equation}\label{etas}
\eta^{(r)}_{v}(z) = \sum_{e\in E_v}\int_{z_e\in \gamma_e}K_v(z,z_e)\cdot I^*_e\xi_{-e}^{(r-1)}(z_e).
\end{equation}

\noindent The main result of \cite{HN} (Theorem 1.1) constructs the desired differential $\Omega_{\ul{s}}$ from the collection $\{\eta_v\}_{v\in V(\Gamma_C)}$:

\begin{theorem}[Hu, Norton \cite{HN}]\label{Thm1.1}
    The differentials $\Omega_{v,\ul{s}} = \Omega_v + \eta_v$ for $v\in V(\Gamma_C)$ glue to a global holomorphic differential $\Omega_{\ul{s}}$ on $C_{\ul{s}}$. Moreover, the $\eta_v$ are normalized by $\int_{A_i} \eta_v = 0$ for $i = 1,\dots ,g(v)$, giving $\Omega_{\ul{0}} = \Omega$.
\end{theorem}  

\begin{remark}
    Theorem 1.1 in \cite{HN} additionally gives a $L^2$-bound for the differentials $\eta_v$. The $A$-normalization also gives uniqueness of the differentials $\eta_v$ such that $\Omega_v + \eta_v$ glue to a global differential \cite[\S 2.4]{HN}.
\end{remark}

Consider a pair of irreducible components $C_v$, $C_{v'}$ of $C$ for $v,v'\in V(\Gamma_C)$. Let $\{A_{i,\ul{s}}, B_{i,\ul{s}}\}_{1\leq i \leq g(v)}$ be extensions to $H_1(C_{\ul{s}},\mathbb{Z})$ of a symplectic basis $\{A_i,B_i\}_{1\leq i \leq g(v)}$ for $H_1(C_{v},\mathbb{Z})$ as in \cite[\S 2.4, \S 4.2]{HN}. Let $\{\nu_i\}_{1\leq i \leq g(v')}$ be a basis of holomorphic differentials for $C_{v'}$ and let $\{\nu_{i,\ul{s}}\}_{1\leq i \leq g(v')}$ be their extensions to $C_{\ul{s}}$ as in Theorem \ref{Thm1.1}. Denote by $\Pi_{\ul{s}}(v',v)= \Pi_{\ul{s}}(v,v')^T$ the ${(g(v') \times g(v))}$-block of entries of the period matrix $\Pi_{\ul{s}}$ of $J(C_{\ul{s}})$ corresponding to the integrals of $\{\nu_{i,\ul{s}}\}_{1\leq i \leq g(v')}$ over the cycles $\{B_{i,\ul{s}}\}_{1\leq j \leq g(v)}$.

\subsection{Local equations}\label{1.1}
To prove Theorem \ref{refinedthm}, we will make use of the following lemmas:
\begin{lemma}\label{lem2}
    For any given point $p$ on a smooth curve $C$ of positive genus, there is a global holomorphic differential restricting to $1\in \omega_C|_p \cong \C$.
\end{lemma}

Lemma \ref{lem2} equivalently says that $\omega_C$ is basepoint-free for $g(C)\geq 1$, which follows from the Riemann-Roch theorem.

\begin{lemma}\label{lem1}
    Let $v,v'\in V(\Gamma_C)_{>0}$. Label the unique shortest sequence of edges connecting $v$ to $v'$ by $(e_1,\dots,e_r)$. The block $\Pi_{\ul{s}}(v',v)$ consists of entries of the form $s_{e_1}\cdots s_{e_r}\cdot p_i$ for some power series $(p_i)_{1\leq i \leq g(v')\times g(v)}$ in the coordinates of $U$. Moreover, if the remaining vertices on the minimal path between $v$ and $v'$ are of genus $0$, at least one of the $p_i$ will be a unit in the coordinate ring of $U$.
\end{lemma}

\begin{proof} 

Recall the notations $\{A_{i,\ul{s}}, B_{i,\ul{s}}\}_{1\leq i \leq g(v)}$ and $\{\nu_{i,\ul{s}}\}_{1\leq i \leq g(v')}$ from Section \ref{term}. Assume that we vary the differential $\Omega\in H^0(\omega_C)$ which corresponds to $\nu_i$ in $H^0(\omega_{C_{v'}})$ and $0$ in $H^0(\omega_{C_{v''}})$ for every $v''\neq v'$. We first explain why the formula in Proposition 3.4 in \cite{HN}, with \textit{our} choices of $v$, $v'$ and $\Omega$, can be refined to
\begin{equation}\label{refinement}
\begin{cases}
\begin{aligned}
    \eta^{(r')}_{v}(z) = & (-1)^{r'}\sum_{l^\in L_{v}^{r'}}s(l)b_{v}(z,q_{e_1})\beta(l)\tilde{\xi}_{-e_{r'}}\\
    &+ s_1\cdots s_{r}\mathcal{O}(|\ul{s}|_\infty^{r'-r+1})&\text{ if } &r'\geq r\\
    \eta^{(r')}_{v}(z) = & 0 &\text{ if } &r'<r.\\
\end{aligned}
\end{cases}
\end{equation}
Next, we show that we can find $i,j$ such that the integral of $\eta^{(r)}_{v}(z)$ over $B_{j,\ul{s}}$ is a product of $s_{e_1}\cdots s_{e_r}$ and a unit. These facts will together with the expressions $\Omega_{v,\ul{s}} = \Omega_v + \sum_{r'\geq 1}\eta_v^{(r')}$ and $\Pi_{\ul{s}}(v',v)_{i,j} = \int_{B_{j,\ul{s}}} \Omega_{v,\ul{s}}$ give the desired result. 

The relation $\eta^{(r')}_{v}(z) = 0 \text{ if } r'< r$ holds since $\xi^{(0)}_e(z_e)$ is zero away from $v(e) = v'$, and inductively by (\ref{xis}), $\xi^{(r'')}_e(z_e)$ is zero when there are strictly more than $r''$ edges adjoining $v(e)$ to $v'$. The relation for $\eta^{(r')}_{v}(z)$ when $r'\geq r$ can be shown by induction on $r$ as in the proof of Proposition 3.4 in \cite{HN}: We first show, for a fixed $e \in E_{v}$, that
\begin{equation*}
\begin{aligned}
    \xi^{(r')}_{e}(z_e) = (-1)^{r'}\sum_{l^\in L_{v}^{r'}}s(l)\boldsymbol{b}_{v}(z_e,q_{e_1})\beta(l)\tilde{\xi}_{-e_{r'}}
    + s_1\cdots s_{r}\mathcal{O}(|\ul{s}|_\infty^{r'-r+1}).
\end{aligned}
\end{equation*}
For $r=1$, we have 
\begin{equation*}
\begin{aligned}\xi_{e}^{(1)}(z_{e}) &= \int_{\zeta_{e_1}\in \gamma_{e_1}} \textbf{K}_{v}(z_e,\zeta_{e_1}) I^*_{e_1}\xi^{(0)}_{-e_1}(\zeta_{e_1})\\
&=-s_{e_1}\boldsymbol{b}_{v}(z_e,q_{e_1})\tilde{\xi}_{-e_1} + \mathcal{O}(|s_{e_1}|^2)
\end{aligned}
\end{equation*}
\noindent as in \cite[Equation 3.16]{HN}. The result for $\eta_{v}^{(1)}(z)$ is obtained analogously. Assume that $r\geq 2$. By the induction hypothesis, we have
\begin{equation*}
\begin{aligned}
    \xi_{-e_1}^{(r-1)}(z_{-e_1}) =& (-1)^{r-1}\left(\prod_{2\leq i \leq r}s_{e_i}\right)\boldsymbol{b}_{v(-e_1)}(z_{-e_1},q_{e_2})\left(\prod_{2\leq i \leq r-1}\beta_{-e_i,e_{i+1}}\right)\tilde{\xi}_{-e_{r}}\\
    &+ s_2\cdots s_{r}\mathcal{O}(|\ul{s}|_\infty).
\end{aligned}
\end{equation*}
Noting that the term $s_2\cdots s_{r}\mathcal{O}(|\ul{s}|_\infty)$ is a holomorphic differential in $z_{-e_1}$, taking $I^*_{e_1}$ of this term (followed by multiplication by $\mathbf{K}_{v}(z_{e},\zeta_{e_1})$ and integration over $\zeta_{e_1}\in \gamma_{e_1}$) will yield an additional factor $s_1$, as desired. We conclude the induction for $\xi_e^{(r')}(z_e)$ when $r'=r$ by noting, as in \cite[Equation 3.17]{HN}, that
\begin{equation*}
\begin{aligned}
\int_{\zeta_{e_1}\in \gamma_{e_1}}\textbf{K}_{v}(z_e,\zeta_{e_1})I_{e_1}^*(\boldsymbol{b}_{v(-e_1)}(z_{-e_1},q_{e_2})) =-s_{e_1}b_{v}(z_e,q_{e_1})\beta_{-e_1,e_2} + \mathcal{O}(|s_{e_1}|^2).\\
\end{aligned}
\end{equation*}
By considering $\xi_e^{(r-1)}(z_e)$ and using (\ref{etas}), we obtain the desired formula for $\eta^{(r)}_{v_l}(z)$. A similar induction gives the formula for $r'>r$, noting that all paths of length $r'>r$ between $v$ and $v'$ contain the path $(e_1,\dots,e_r)$.

The coefficients $(\beta_{-e_i,e_{i+1}})_{1\leq i \leq r-1}$ are nonzero since they are evaluations of the fundamental bidifferential $\frac{dwdz}{(w-z)^2}$ of $\pr^1$ at two distinct points. Moreover, the coefficient $\tilde{\xi}_{-e_r}$ equals $\tilde{\nu_i}(q_{-e_r})$ where $\nu_i(z) = \tilde{\nu_i}(z)dz$. By Lemma \ref{lem2}, we can choose $i$ such that $\tilde{\nu}_i(q_{-e_r})\neq 0$. Lastly, we can find a homology cycle $B_j$ on $C_v$ such that $\int_{B_j}b_{v}(z,q_{e_1})\neq 0$. Indeed, we have $\int_{B_j}b_{v}(z,q_{e_1}) = \tilde{\nu}'_j(q_{e_1})$ where $\nu'_j$ is the holomorphic differential of $C_v$ dual to $A_j$ \cite[p.28]{HN}. By Lemma \ref{lem2} there is a $j$ such that $\tilde{\nu}'_j(q_{e_1})\neq 0$. For $i,j$ as above, the integral over $B_j$ of the $s_{e_1}\cdots s_{e_r}$-term of (\ref{refinement}) is thus nonzero as desired.\end{proof}

\subsubsection{Proof of Theorem \ref{refinedthm}}
\begin{proof}[Proof of Theorem \ref{refinedthm}]
Let $\bt t^*\colon \C[W] \to \C[U],\et$ $\bt\hspace{-8pt} \pi^*\colon \C[W] \to \C[V]\et$ denote the pullbacks on coordinate rings associated to $t,\pi$. By Lemma \ref{lem1}, the tensor product of the coordinate rings $\C[U]$, $\C[V]$ over $\C[W]$ with respect to $t^*,\pi^*$ is generated by the variables $x_{i,j}$ and $x^{(i')}_{m,n}$ with the following relations:

\begin{enumerate}[label = (\arabic*)]
    \item For each variable $x^{(i')}_{m,n}$ of $\C[V]$ and $y_{i,j}\in \C[W]$ the unique variable satisfying ${\pi^*y_{i,j} = x^{(i')}_{m,n}}$, we have the relation $x^{(i')}_{m,n} = t^*y_{i,j}$.
 
    \item For each pair of vertices $v,v'\in V(\Gamma_C)_{>0}$ such that $\sigma(v)\neq \sigma(v')$, we obtain the relation $s_{e_1}\cdots s_{e_r}=0$ for edges $e_1,\dots,e_r$ as in Lemma \ref{lem1}. Indeed, the path $e_1,\dots,e_r$ contains a critical subpath $e_1',\dots, e_t'$. By Lemma \ref{lem1} we obtain the relation ${s_{e'_1}\cdots s_{e'_t}\cdot \text{unit} = 0}$ and hence $s_{e_1}\cdots s_{e_r}=0$.

\end{enumerate}

Note that, to obtain a minimal set of relations of type $(2)$ in the fiber product, we only need to consider the relations $s_{e_1}\cdots s_{e_r}=0$ for critical paths $(e_1,\dots,e_r)$.

Relations of type (1) eliminate the generators $\ul{x}^{(i')}$ for $1\leq i'\leq k$, while the relations of type (2) are the square-free monomials $\{\prod_{e\in \gamma}s_e\}_{\gamma \in P_\sigma}$. This gives our desired coordinate ring of the fiber product. Since the monomials $\{\prod_{e\in \gamma}s_e\}_{\gamma \in P_\sigma}$ are square-free, the coordinate ring is reduced.\end{proof}

\section{The Torelli pullback of $[\A_{g_1}\times \dots \times \A_{g_k}]$}\label{a3}

\subsection{Strata of the fiber product}

We give a stratification of the fiber product $Y$ (see Section \ref{fibprod}) according to the combinatorial type of the curve $[C]\in \Mct_g$ and the $k$-coloring $\sigma$.

Let $T$ be a stable tree of genus $g$ (see \cite[A.1]{GP}) and $V(T)_{>0} = \{v\in V(T)| g(v)>0\}$. Let
$$\sigma: V(T)_{>0}\to \{1,\dots,k\}$$\noindent be a partition satisfying $\sum_{v \in \sigma^{-1}(j)} g(v) = g_j$ for $j=1,\dots,k$. For each such pair $(T,\sigma)$, we obtain a map $$\varphi_{(T,\sigma)} \colon \Mct_T = \prod_{v\in V(T)}\Mct_{g(v),n(v)}\to Y$$ via the gluing map $\xi_T\colon \Mct_T \to \Mct_g$ and the map
\begin{equation*}
\begin{aligned}
\Mct_T &\to \A_{g_1}\times\dots\times \A_{g_k}\\
[(C_v,q_1,\dots,q_{n(v)})_{v\in V(T)}] &\mapsto ([\hspace{-8pt}\prod_{v\in \sigma^{-1}(1)}\hspace{-8pt}J(C_v)],\dots,[\hspace{-8pt}\prod_{v\in \sigma^{-1}(k)}\hspace{-8pt}J(C_v)]).
\end{aligned}
\end{equation*}

Denote the image of $\varphi_{(T,\sigma)}$ in $Y$ by $Y_{(T,\sigma)}$. The pair $(T,\sigma)$ is uniquely determined by the closed subscheme $Y_{(T,\sigma)}\subset Y$. Containments of the form $Y_{(T,\sigma)}\subset Y_{(T',\sigma')}$ are determined by specializations $T \to T'$ of stable graphs \cite[A.2]{GP} such that the induced map ${f \colon V(T)\to V(T')}$ satisfies $\sigma(v) = \sigma'(f(v))$ for $v\in V(T)_{>0}$. We denote such specializations by ${\phi\colon (T,\sigma) \to (T',\sigma')}$, and refer to $\phi$ as a \textit{$(T',\sigma')$-structure} on $(T,\sigma)$. For a collection $S = \left((T_i,\sigma_i)\right)_{1\leq i \leq n}$, we refer to a collection $\left(\phi_i\colon (T,\sigma)\to (T_i,\sigma_i)\right)_{1\leq i \leq n}$ as an \textit{$S$-structure} on $(T,\sigma)$, denoted $\left((T,\sigma),(\phi_i)_{1\leq i \leq n}\right)$. We say that an $S$-structure $\left((T,\sigma),(\phi_i)_{1\leq i \leq n}\right)$ is \textit{generic} if, for any $S$-structure $\left((T',\sigma'),(\phi'_i)_{1\leq i \leq n}\right)$ and specialization $\phi\colon (T,\sigma)\to (T',\sigma')$ with $\phi'_i\circ \phi = \phi_i$ for $1\leq i \leq n$, the map $\phi$ is an isomorphism.

\begin{proposition}\label{unique}
    Let $S = \left((T_i,\sigma_i)\right)_{1\leq i \leq n}$. For any $S$-structure $\left((T,\sigma),(\phi_i)_{1\leq i \leq n}\right)$, there is a unique generic $S$-structure $\left((T',\sigma'),(\phi'_i)_{1\leq i \leq n}\right)$ and specialization $\phi\colon (T,\sigma)\to (T',\sigma')$ satisfying $\phi'_i\circ \phi = \phi_i$ for $1\leq i \leq n$.
\end{proposition}

\begin{proof}
    Given $\left((T,\sigma),(\phi_i)_{1\leq i \leq n}\right)$, the unique generic $S$-structure is obtained by contracting the edges of $T$ which are contracted under the specializations $T\to T_i$ induced by $\phi_i$ for all $1\leq i \leq n$, with the induced coloring and $S$-structure.
\end{proof}

Let $\text{Irr}(Y)$ be the set of irreducible components of $Y$. Each $Z \in \text{Irr}(Y)$ has the form $Y_{(T,\sigma)}$ for some pair $(T,\sigma)$ admitting no $(T',\sigma')$-structure for $(T',\sigma')\neq (T,\sigma)$. For any subset $\mathcal{Z} \subset \text{Irr}(Y)$, write $S_\mathcal{Z} = \{(T,\sigma)|Y_{(T,\sigma)}\in \mathcal{Z}\}$ and let $G_\mathcal{Z}$ be the set of generic $S_\mathcal{Z}$-structures.

\begin{proposition}
    The fiber product of the maps $(\varphi_{(T,\sigma)} \colon\Mct_T\to Y)_{(T,\sigma)\in S_\mathcal{Z}}$ is canonically isomorphic to $\bigsqcup_{((T',\sigma'),(\phi_Z)_{Z\in \mathcal{Z}})\in G_\mathcal{Z}}\Mct_{T'}$.
\end{proposition}

\begin{proof}
    The fiber product of the maps $(\varphi_{(T,\sigma)})_{(T,\sigma)\in S_\mathcal{Z}}$ can be written as $\bigcup_{((T',\sigma'),(\phi_Z)_{Z\in \mathcal{Z}})\in G_\mathcal{Z}}\Mct_{T'}$. Proposition \ref{unique} ensures that this union is disjoint. \end{proof}  

Define $\text{Strata}(Y)$ to be the set of closed subschemes $Y_{(T,\sigma)}$ such that $(T,\sigma)$ admits a generic $S_\mathcal{Z}$-structure for some $\mathcal{Z}\subset \text{Irr}(Y)$. For $Y_{(T,\sigma)}\in \text{Strata{(Y)}}$, denote by $Y^\circ_{(T,\sigma)}$ the open subscheme of $Y_{(T,\sigma)}$ which is the complement of the closed subschemes $Y_{(T',\sigma')}\subsetneq Y_{(T,\sigma)}$ with $Y_{(T',\sigma')} \in \text{Strata}(Y)$. The set $$\text{Strata}^\circ(Y) = \{Y^\circ_{(T,\sigma)}| Y_{(T,\sigma)} \in \text{Strata}(Y)\}$$ forms a locally closed stratification of $Y$.

\begin{remark}
    A complete description of the strata of the fiber product $Y$ of the Torelli map with a general product map $\A_{g_1}\times\dots\times \A_{g_k} \to \A_g$ will be provided in \cite{inprep}. A combinatorial stratum will consist of a stable tree $T$ of genus $g$ and a $k$-coloring of $V(T)_{>0}$, such that each edge of $T$ lies on a path between two vertices of different colors where the remaining vertices on the path are all of genus $0$.
\end{remark}

\subsection{Proof of Theorem \ref{vanishing}}

\begin{proof}[Proof of Theorem \ref{vanishing}]
By Theorem \ref{mainthm}, the fiber product $Y$ is reduced. Using excess intersection theory \cite{Ful}, we can thus decompose the class $\alpha = t^*[\A_{g_1}\times\dots\times \A_{g_k}]$ as $$\alpha = \sum_{Y_{(T,\sigma)}\in \text{Strata}(Y)}{\xi_T}_*\alpha_{(T,\sigma)}$$ where $\alpha_{(T,\sigma)}$ is supported on $Y_{(T,\sigma)}$ (see \cite[Theorem 5]{inprep}). Moreover, we can take $\alpha_{(T,\sigma)}$ to be a polynomial in the Chern classes of the normal bundles of $Y_{(T,\sigma)}$ in $Y_{(T',\sigma')}$ associated to $(T,\sigma)\to (T',\sigma')$ and of $\A_{g_1}\times\dots\times \A_{g_k}$ in $\A_g$.

The normal bundles are given by $$N_{\A_{g_1}\times\dots\times \A_{g_k}/\A_g} = S^2(\bigoplus_{1\leq i\leq k}\mathbb{E}_{g_i}^\vee) - \bigoplus_{1 \leq i \leq k} S^2\mathbb{E}_{g_i}^\vee = \bigoplus_{1\leq i < j \leq k} \mathbb{E}_{g_i}^\vee\boxtimes\mathbb{E}_{g_j}^\vee$$ and $$N_{Y_{(T,\sigma)}/Y_{(T',\sigma')}} = N_{\Mct_T/\Mct_{T'}} = \bigoplus_{e\in E(T)\backslash E(T')} T_{h_e}\otimes T_{h'_e}$$\noindent where $h_e$, $h_e'$ are the half-edges of $e$ and $T_{h_e},T_{h_e'}$ are the tangent line bundles associated to $h_e,h_e'$ on the relevant factor of $\Mct_{T}$ \cite[\S 13.3]{ACG}.

Taking Chern classes of these normal bundles, we deduce that the class $\alpha_{(T,\sigma)}$ is tautological on $Y_{(T,\sigma)} = \M_T^{ct}$. Pushing forward under the gluing map, the resulting class is tautological on $\Mct_g$.\end{proof}

Recall from Section \ref{pullbacks} that, for $d = \text{codim}(t^*[\A_{g_1}\times\dots\times \A_{g_k}])$, we have $$t^*[\A_{g_1}\times\dots\times \A_{g_k}]=0 \hspace{5pt}\text{for}\hspace{5pt}d>2g-3.$$
\noindent We now exhibit the possibilities for $g_1,\dots,g_k$ such that $d\leq 2g-3$.

\begin{namedtheorem*}{Proposition \ref{tuples}}
   The tuples $(g_1,\dots,g_k)$ with $1\leq g_1\leq \dots \leq g_k$ and $\sum_{1\leq i \leq k}g_i = g$ such that $d= \text{codim}(t^*[\A_{g_1}\times\dots\times \A_{g_k}])\leq 2g-3$ are $(1,g-1),(1,1,g-2),(2,g-2)$ and $(3,3)$. 
\end{namedtheorem*}

\begin{proof}
When $g_1 = 1$, $k\geq 3$ and $g_2\geq 2$, the codimension $d$ is minimized when $k=3$ and $g_2 =2$. Using $d=\sum_{1\leq i <j \leq k} g_ig_j$, we obtain $d\geq  1(g-1) + 2(g-3) = 3g-7 > 2g-3$. When $g_1 = g_2 = 1$ and $k\geq 4$, $d$ is minimized for $k=4$ and $g_3 = 1$, giving $d \geq 3g-6 > 2g-3$. This leaves the possibilities $(1,g-1)$, $(1,1,g-2)$ for $g_1=1$. When $g_1 \geq 2$, $k\geq 3$, the codimension $d$ is minimized for $k=3$ and $g_1=g_2=2$. We obtain $d\geq 2(g-2) + 2(g-4) = 4g-12 >2g-3$. This leaves the case $(2,g-2)$ for $g_1 =2$. For $g_1 = 3$, $k=2$, we have $d = 3(g-3).$ Thus, the only possibility for $g$ such that $d\leq 2g-3$ is $g=6$. This corresponds to the tuple $(3,3)$. For $g_1\geq 4$, $d$ is minimized when $k=2$ and $g_1 = 4$. Thus $d\geq 4(g-4) > 2g-3$. 
\end{proof}

\bibliographystyle{amsplain}
\bibliography{biblio}

Lycka~Drakengren\par\nopagebreak
\textsc{Department of Mathematics, ETH Zürich}\par\nopagebreak
\textsc{Rämistrasse 101, 8092 Zürich, Switzerland}\par\nopagebreak
\textit{E-mail address}: \texttt{lycka.drakengren@math.ethz.ch}
\end{document}

%% file: criticalpaths.tex
\tikzset{every picture/.style={line width=0.75pt}} 

\begin{tikzpicture}[x=0.75pt,y=0.75pt,yscale=-1,xscale=1]

\draw [line width=1.5]    (329.29,118.13) -- (322.71,153.42) ;
\draw [line width=1.5]    (322.71,157.71) -- (382.29,259.13) ;
\draw [line width=1.5]    (322.71,157.71) -- (276.29,218.13) ;
\draw  [fill={rgb, 255:red, 0; green, 0; blue, 0 }  ,fill opacity=1 ] (318.42,157.71) .. controls (318.42,155.34) and (320.34,153.42) .. (322.71,153.42) .. controls (325.08,153.42) and (327,155.34) .. (327,157.71) .. controls (327,160.08) and (325.08,162) .. (322.71,162) .. controls (320.34,162) and (318.42,160.08) .. (318.42,157.71) -- cycle ;
\draw [line width=1.5]    (330.29,36.13) -- (329.29,118.13) ;
\draw  [fill={rgb, 255:red, 186; green, 27; blue, 30 }  ,fill opacity=1 ] (321.42,118.13) .. controls (321.42,113.78) and (324.94,110.25) .. (329.29,110.25) .. controls (333.64,110.25) and (337.17,113.78) .. (337.17,118.13) .. controls (337.17,122.47) and (333.64,126) .. (329.29,126) .. controls (324.94,126) and (321.42,122.47) .. (321.42,118.13) -- cycle ;
\draw  [fill={rgb, 255:red, 3; green, 145; blue, 175 }  ,fill opacity=1 ] (322.42,36.13) .. controls (322.42,31.78) and (325.94,28.25) .. (330.29,28.25) .. controls (334.64,28.25) and (338.17,31.78) .. (338.17,36.13) .. controls (338.17,40.47) and (334.64,44) .. (330.29,44) .. controls (325.94,44) and (322.42,40.47) .. (322.42,36.13) -- cycle ;
\draw  [fill={rgb, 255:red, 3; green, 145; blue, 175 }  ,fill opacity=1 ] (268.42,218.13) .. controls (268.42,213.78) and (271.94,210.25) .. (276.29,210.25) .. controls (280.64,210.25) and (284.17,213.78) .. (284.17,218.13) .. controls (284.17,222.47) and (280.64,226) .. (276.29,226) .. controls (271.94,226) and (268.42,222.47) .. (268.42,218.13) -- cycle ;
\draw  [fill={rgb, 255:red, 3; green, 145; blue, 175 }  ,fill opacity=1 ] (374.42,259.13) .. controls (374.42,254.78) and (377.94,251.25) .. (382.29,251.25) .. controls (386.64,251.25) and (390.17,254.78) .. (390.17,259.13) .. controls (390.17,263.47) and (386.64,267) .. (382.29,267) .. controls (377.94,267) and (374.42,263.47) .. (374.42,259.13) -- cycle ;
\draw   (309.42,77.25) .. controls (309.42,58.25) and (314.42,15.25) .. (331.42,14.25) .. controls (348.42,13.25) and (350.42,56.25) .. (351.42,75.25) .. controls (352.42,94.25) and (348.42,133.25) .. (331.42,134.25) .. controls (314.42,135.25) and (309.42,96.25) .. (309.42,77.25) -- cycle ;
\draw   (398.42,285.25) .. controls (377.42,295.25) and (333.42,199.25) .. (322.42,183.25) .. controls (311.42,167.25) and (307.42,164.25) .. (307.42,149.25) .. controls (307.42,134.25) and (314.42,82.25) .. (335.42,96.25) .. controls (356.42,110.25) and (340.42,143.25) .. (338.42,152.25) .. controls (336.42,161.25) and (345.42,178.25) .. (355.42,192.25) .. controls (365.42,206.25) and (419.42,275.25) .. (398.42,285.25) -- cycle ;
\draw   (258.42,237.25) .. controls (247.42,227.25) and (276.42,197.25) .. (285.42,186.25) .. controls (294.42,175.25) and (310.42,165.25) .. (312.42,151.25) .. controls (314.42,137.25) and (315.42,94.25) .. (334.42,105.25) .. controls (353.42,116.25) and (334.42,147.25) .. (332.42,156.25) .. controls (330.42,165.25) and (316.42,189.25) .. (307.42,201.25) .. controls (298.42,213.25) and (269.42,247.25) .. (258.42,237.25) -- cycle ;

\end{tikzpicture}